\documentclass{elsarticle}

\usepackage{amsmath}
\usepackage{amssymb}
\usepackage{amsthm}
\usepackage{algorithm}
\usepackage{algpseudocode}
\usepackage{graphicx}
\usepackage{enumerate}
\usepackage{float}
\usepackage{subcaption}
\usepackage{mathtools}

\DeclareMathOperator{\diam}{diam}

\DeclareMathOperator{\N}{\mathbb{N}}
\DeclareMathOperator{\R}{\mathbb{R}}
\DeclareMathOperator{\Part}{\mathcal{P}}

\let\Re\undefined
\DeclareMathOperator{\Re}{\text{Re}}
\let\Im\undefined
\DeclareMathOperator{\Im}{\text{Im}}

\DeclarePairedDelimiter{\ceil}{\lceil}{\rceil}

\newtheorem{theorem}{Theorem}[section]
\newtheorem{lemma}[theorem]{Lemma}
\newtheorem{corollary}[theorem]{Corollary}
\theoremstyle{definition}
\newtheorem{definition}[theorem]{Definition}

\newtheorem{fact}[theorem]{Fact}

\newtheorem{question}[theorem]{Question}

\newtheorem*{acknowledgement*}{Acknowledgement}
\theoremstyle{remark}
\newtheorem{remark}[theorem]{Remark}
\numberwithin{equation}{section}

\begin{document}
	
	\title{H\"older functions with nonporous graphs}
	\author[rvt]{Claudio A. DiMarco}
	\address[rvt]{1 Lomb Memorial Drive, School of Mathematics and Statistics, Rochester Institute of Technology, Rochester, NY 14623, USA}
	\ead{cdsma@rit.edu}
	\date{\today}

\begin{abstract}
For any $\alpha \in (0,1)$ there is a H\"older continuous function $h:\R \rightarrow \R$ with nonporous graph $Gr(h) \subset \mathbb{R}^2$ which is thin for doubling measures.  Construction of these functions requires analysis of a known example with nonporous graph and bounded variation on the unit interval.
\end{abstract}

\maketitle

\section{Introduction}

Certain subsets $E$ of a metric space $X$ are ``thin" or ``fat" for doubling measures.  Wu initiated the discussion of thinness for doubling measures on $\R$ in \cite{Wu}, wherein the relationship between thinness and ``porosity" was also investigated. The results in \cite{Wu} employ a generalized notion of porosity which easily implies thinness for doubling measures, as Chen and Wen noted in \cite[page 9]{ChenWen}.  In \cite{DiMarco} we proved that the graph of any Zygmund function on $\R^d$ is porous and therefore thin in $\R^{d+1}$, where some of the following discussion was included.

In \cite{Ojala2} the authors considered the case when $X$ is uniformly perfect and $E \subset X$ has the property $\mu(E) = 0$ for every doubling measure $\mu$ on $X$, or the property $\mu(E) > 0$ for all doubling measures. Such subsets are called thin and fat, respectively.  Ojala et al provided sufficient conditions for certain ``cut-out sets" being thin or fat, and proved that $E \subset X$ is thin if and only if $E$ is quasisymmetrically null, i.e. $\mathcal{H}^q(f(E)) = 0$ whenever $f:X \rightarrow Y$ is quasisymmetric and $Y$ is Ahlfors $q$-regular, where $\mathcal{H}^q$ is $q$-dimensional Hausdorff measure \cite{Ojala2}.  

It is interesting to consider the special case when $E = Gr(f) \subset \R^{d+1}$ is the graph of a continuous function $f$.  This is an area of recent interest in which relatively little is understood.  For a while, an open question in this area was whether every rectifiable curve in the plane is thin for doubling measures. Surprisingly, this turned out not to be the case (see \cite{Garnett}).  However, if attention is restricted to isotropic doubling measures on $\R^{d+1}$, then indeed $Gr(f)$ has measure zero for any continuous $f:[0,1]^d \rightarrow [0,1]$, as seen in \cite{ChenWen}.  This is easily extended to include all continuous $f: \R^d \rightarrow \R$.

\subsection{H\"older continuity and porosity}
Recall that a bounded function $f: \R^d \rightarrow \R$ is \textit{H\"older continuous} if there are $0<\alpha \leq 1$ and $C>0$ such that $|f(x) - f(y)| \leq C|x-y|^{\alpha}$ for all $x,y \in \R^d,$ and is called \textit{Lipschitz} if there is $L>0$ such that $|f(x) - f(y)| \leq L|x-y|$ for all $x,y \in \R^d$.  Clearly $f$ is also $\alpha'$-H\"older when it is $\alpha$-H\"older, for any $\alpha' \in (0, \alpha)$.  We only consider $\alpha \leq 1$ because $f$ is constant if $\alpha > 1$, which is readily seen by noting $f' = 0$.  In \cite{Ojala} the authors showed that $Gr(f)$ need not be thin in general, even for continuous $f$.  They remark that there are few known sufficient conditions for thinness. In particular, it is known that the graph of a Lipschitz function is thin, but whether this property extends to H\"older functions remains unknown \cite[Question 1.2]{Ojala}.

In \cite{Goffman} Goffman constructed a continuous function on $[-1,1]$ consisting of a sequence of spikes, which has bounded $p$-variation for $p \geq 1$.  We provide a more detailed description of this function and generalize it to obtain a family $\mathcal{H}$ (defined in Section \ref{main_section}) of real-valued functions on the unit interval with similar properties.  It is easy to see $Gr(h)$ is thin for any $h \in \mathcal{H}$ because it is locally Lipschitz by construction.

For any $h \in \mathcal{H}$, $Gr(h)$ is thin yet nonporous.  The existence of such functions demonstrates that not every H\"older function has a porous graph, answering a question posed in \cite{DiMarco}.  Moreover, for any prescribed $\alpha' \in (0, 1)$ there is $h \in \mathcal{H}$ that is $\alpha'$-H\"older but not $\alpha$-H\"older for any $\alpha > \alpha'$.  We adopt the following definition of porosity found in \cite{ChenWen}.
\begin{definition}\label{porosity_def1}
    A subset $E$ of a metric space $X$ is called \textit{porous} if there exists $a \in (0,1)$ such that for all $x\in E$ and $r>0$, there is a ball $B(y,ar) \subset B(x,r)$ with $B(y, ar) \cap E = \varnothing$.
\end{definition}
\noindent In \cite{Goffman} porosity is defined pointwise using oriented cubes (with sides parallel to the coordinate axes in $\R^d$) instead of balls, where $Q_x$ denotes the oriented cube centered at $x$ with Lebesgue measure $|Q_x|$.
\begin{definition}\label{porosity_def2}
    Let $E\subset \R^d, ~x \in E$, and define 
    \begin{equation*}
        \sigma(Q_x) = \sup \{|Q_y| : Q_y\subset Q_x \text{ and } Q_y \cap E = \varnothing\}.
    \end{equation*}
    The set $E$ is \textit{porous at} $x$ if 
    \begin{equation}\label{porous_at_x}
        s(x) = \limsup_{|Q_x| \rightarrow 0} \frac{\sigma(Q_x)}{|Q_x|} > 0.
    \end{equation}
\end{definition}
Observe that the full porosity of $E$ (in the sense of Definition \ref{porosity_def1}) implies local porosity in the sense of Definition \ref{porosity_def2}.  Unsurprisingly porosity can be defined in terms of cubes instead of balls, so we make that transition in the proof of the following fact, and in Lemma \ref{cubes_lemma} which is used to prove the main result.

\newpage
\begin{fact}
If $E \subset \R^d$ is porous, then it is porous at every $x \in E$.
\end{fact}
\begin{proof}
    Assume Definition \ref{porosity_def1} is satisfied and let $x \in E$.  Let $a>0$ be as in Definition \ref{porosity_def1}, and let $Q_x(n)$ denote the cube centered at $x$ with side length $2/n.$  Then there is a ball $B(y_n, a/n) \subset B(x, 1/n)$ such that $B(y_n, a/n) \cap E = \varnothing$, and the cube $Q_{y_n}$ inscribed in the ball $B(y_n, a/n)$ has 
    \[
    |Q_{y_n}| = \left( \frac{2a}{n\sqrt{d}} \right)^2 = \frac{4a^2}{n^2 d}
    \]
    Therefore $\sigma(Q_x(n)) \geq 4a^2/(n^2 d)$ and hence
    \[
    s(x) \geq \frac{\sigma(Q_x(n))}{|Q_x(n)|} \geq \frac{4a^2/(n^2 d)}{4/n^2} = a^2/d > 0. \qedhere
    \]
\end{proof}

\begin{lemma}\label{cubes_lemma}
    Let $E \subset \R^d$.  Suppose there is a constant $C>0$ and a sequence of cubes $Q_n \subset \R^d$, each centered at some $x_n \in E$, such that for any cube $Q\subset (Q_n \setminus E), ~l(Q) \leq (C/n)l(Q_n)$.  Then $E$ is not a porous subset of $\R^d$.
\end{lemma}
\begin{proof}
    Suppose $E$ is porous with constant $a>0$ such that any ball $B(x,r)$ with $x\in E$ contains a ball $B(y, ar)$ that does not meet $E$.  Let $B_n$ be the ball inscribed in $Q_n$.  By porosity of $E$, we may choose $z_0 \in \R^d$ such that $B_n' = aB_n + z_0 \subset (B_n \setminus E)$.  The cube $Q_n'$ inscribed in $B_n'$ has the property
    \begin{equation}\label{cubes_no_porosity}
        l(Q_n') \leq (C/n)l(Q_n).
    \end{equation}
    Denote $r_n', r_n$ the radii of $B_n', B_n$ respectively, so that $r_n' = ar_n$ and $l(Q_n') = \sqrt{2}r_n' = \sqrt{2}a r_n.$  Also $l(Q_n) = r_n$, so by \eqref{cubes_no_porosity}
    \begin{equation*}
        \sqrt{2}a r_n \leq (C/n) r_n,
    \end{equation*}
    which fails for large $n$.
\end{proof}

\subsection{Connection between porosity and Assouad dimension}
Porosity is closely related to the notion of Assouad dimension, as Fraser discussed in \cite[pp. 10, 13, 73-74]{Fraser}, so we include some relevant details.  Definition \ref{porosity_def1} says that a porous set $F \subset \R^d$ has uniformly large holes at all locations and scales.  The definition and geometric interpretation of Assouad dimension are nontransparent.  There are many equivalent versions; for example, the following is found in \cite[p.10]{Fraser}.
\begin{definition}\label{Assouad_dim_def}
    The \textit{Assouad dimension} of a non-empty set $F\subset \R^d$ is
    \begin{equation}
      \begin{split}
        \dim_A F = \inf \Bigl\{ \alpha : \; & \text{there exists a constant } C > 0 \text{ such that,} \\
        & \text{for all } 0 < r < R \text{ and } x \in F, \\
        & N_r(B(x, R) \cap F) \leqslant C\!\left(\frac{R}{r}\right)^{\!\alpha} \Bigr\},
      \end{split}
    \end{equation}
    where $N_r(E)$ is the smallest number of open sets required for an $r$-cover of a bounded set $E$.
\end{definition}

Suppose we can cover any $R$-ball in a set $F$ with a constant multiple (independent of $x, r, R$) of $(R/r)^{\alpha}$ many $r$-balls. This is natural since if $F$ is $\mathbb{R}^d$ then one readily covers any $R$-ball by at most $5 \cdot 2^d (R/r)^d$ many $r$-balls. This becomes harder to satisfy as $\alpha$ decreases, and the Assouad dimension is simply the infimum of $\alpha$ such that the condition is satisfied \cite[p.13]{Fraser}.  Geometrically, $F \subset \R^d$ has $\dim_A F < n$ means that $F$ is significantly less diffuse than balls $B(z,R) \subset \R^d$ at all locations $z\in F$ and scales $R$, and hence $F\cap B(z, R)$ is easier to cover than $B(z,R)$ with smaller balls of radius $r.$

Luukkainen proved the following result \cite[Theorem 5.2]{Luukkainen2} which clearly implies that $F \subset \R^d$ is porous if and only if $\dim_A F < d$. 
\begin{theorem}[Luukkainen]
    There is a constant $\varepsilon(a, d) >  0$ depending only on $a$ and $d$ such that if $F \subset \R^d$ is porous with constant $a$, then 
    \begin{equation*}
        \dim_A F \leq d - \varepsilon(a, d).
    \end{equation*}
\end{theorem}

The Assouad spectrum was introduced by Fraser and Yu in \cite{FY}; it is a one-parameter family of metrically defined dimensions which interpolates between the upper box-counting dimension and the (quasi-)Assouad dimension \cite{CT}.  Recently Chrontsios-Garitsis and Tyson studied how the $\alpha$-H\"older condition on a function $f:[a,b] \rightarrow \R$ influences bounds for the Assouad spectra of $Gr(f)$.  They provided upper bounds on the regularized Assouad spectra of $Gr(f)$, and an algorithm to construct a related function $\Tilde{f}:[a,b] \rightarrow \R$ that realizes these upper bounds \cite[Theorems 1.1 and 1.2]{CT}.  

\subsection{Organization of the paper}
The graph of Goffman's function, constructed explicitly in Section \ref{goffman_section}, is indeed nonporous in the sense of Definition \ref{porosity_def1}, but this property is achievable with a less complicated function, as demonstrated via construction of the family $\mathcal{H}$ in Section \ref{main_section}. 

Section \ref{main_section} contains the main result: we prove that H\"older continuity can be achieved for any prescribed exponent $\alpha \in (0,1)$ by a member of our simplified family $\mathcal{H}$, and demonstrate the lack of porosity of $Gr(h)$ for every member of $\mathcal{H}$.  H\"older continuity follows by application of a common lemma to a standard sum of ``tent functions" on the unit interval, and we disprove porosity via a lemma that echoes the definition of Assouad dimension.

Section \ref{examples_section} contains examples.  It is shown that Goffman's function is H\"older continuous with exponent $\alpha = 1/6$ due to the behavior of the ratio \\$|f(x) - f(y)|/|x-y|^{\alpha}$ on the increasingly steep segments of the graph.  We also provide a small family of Goffman-type functions $\mathcal{H}_{sm}$ that share the zig-zag construction but include different spike heights and frequencies, attaining all possible H\"older exponents in $(0, 1/4].$  For sake of variety we provide a constructive proof that the sparser (in comparison to the original Goffman function) curves $Gr(h), ~h \in \mathcal{H}_{sm}$, are nonporous, from which it is intuitively obvious that the graph of the original Goffman function also lacks porosity.

\begin{remark}
    The families $\mathcal{H}$ and $\mathcal{H}_{sm}$ are both introduced because their members are constructed differently (the former with tents, the latter with zig-zags), which results in different behavior.  To attain any prescribed H\"older exponent in $(0,1)$ we use $\mathcal{H}$, while $\mathcal{H}_{sm}$ only provides exponents in $(0, 1/4]$, hence ``small."
\end{remark}

\section{Preliminaries}
Most of the content surrounds functions $f: I \rightarrow \R$ where $I \subset \R$ represents a closed interval of positive length $|I|$ or $I = \R$.  Given a cube $Q\subset \R^2$ we express the side length as $l(Q)$, and all cubes are assumed oriented so that the sides are parallel to the coordinate axes unless otherwise stated.  We also write $l(S)$ to denote the length of a line segment $S \subset \R^2$.  By ``length" we intend the standard metric on $\R^d$, and write $|z|$ for the Euclidean norm of $z \in \R^d$.  

The graph of $f:A \rightarrow B$ is denoted $Gr(f) = \{(x, f(x))~|~ x\in A\}$.  The statement $a \lesssim b$ means there is a constant $C>0$ such that $a \leq Cb$, and the value of $C$ is not important.  We write $a \approx b$ when $a$ is comparable to $b$, that is $cb \leq a \leq c'b$ for some irrelevant constants $c, c' >0$.  The support of a real-valued function $g$ is denoted $supp ~g$.  As usual $\pi_i(x_1, \dots, x_d) = x_i$ for $(x_1, \dots, x_d) \in \R^d$, $i = 1, \dots, n$.

The symbols $B(x,r)$ and $\bar{B}(x,r)$ denote open and closed balls of radius $r$ respectively.  A metric space is called \textit{doubling} if there is $C_1 \geq 1$ so that every set of diameter $d$ in the space can be covered by $C_1$ sets of diameter at most $d/2$ \cite[page 81]{Heinonen}.  Every complete doubling metric space supports a doubling measure \cite{Luukkainen}.  The following definition can be found in \cite[page 3]{Heinonen}:
\begin{definition}\label{doubling_def}
	A Borel measure $\mu$ on a metric space $(X,d)$ is called \textit{doubling} if there is $C>0$ such that $\mu(\bar{B}(x, r)) \leq C \mu(\bar{B}(x, r/2))$ for all nonempty closed balls $\bar{B}(x,r)$.
\end{definition}
A doubling metric space that supports a doubling measure $\mu$ is often denoted $(X, d, \mu)$.  Given such a space, the doubling property  of $\mu$ ensures that the measure is not disproportionately concentrated around any particular point $x\in X$.
\begin{definition}
	A subset $E$ of a metric space $(X,d)$ is \textit{thin} if $\mu(E)=0$ for every doubling measure $\mu$ on $X$.
\end{definition}

\section{Goffman's function}\label{goffman_section}
Working only to the right of the origin in $\R^2$, we reconstruct Goffman's example to provide clarity, along with insight for generalization.  The curves $y = \pm \sqrt{x}$ on $[0,1]$ are used as guides for the construction of the desired function $f:[0,1] \rightarrow \R$.  Define $f(0)= 0$; for $n \in \N$ define $I_n$ as follows, and note the length of $I_n$, written $|I_n|.$
\begin{equation}
    \begin{split}\label{def_In}
        I_n &= \left[ \frac{1}{n+1}, \frac{1}{n} \right] \text{  and  } f\left(\frac{1}{n+1}\right) = f\left(\frac{1}{n}\right) = 0 \\
        |I_n| &= 1/(n(n+1))
    \end{split}
\end{equation}

Within each $I_n$ we construct $2n$ `spikes' while keeping the function continuous.  To this end, consider the uniform partition $\Part_n$ of $I_n$ into $n$ subintervals, each to host two spikes:
\begin{equation}\label{partition_of_In}
    \Part_n = \left\{ \frac{1}{n+1} + \frac{m|I_n|}{n} \right\}_{m=0}^n = \left\{ \frac{1}{n+1} + \frac{m}{n^2 (n+1)} \right\}_{m=0}^n
\end{equation}
Denote the $m$th subinterval of $I_n$ by $I_{n,m}$, so that 
\begin{equation}
    \begin{split}\label{defInm}
        I_{n,m} &= \left[ \frac{1}{n+1} + \frac{m}{n^2 (n+1)}, \frac{1}{n+1} + \frac{m+1}{n^2 (n+1)} \right], ~~m = 0, 1, \dots,  n-1. \\
        |I_{n,m}| &= \frac{|I_n|}{n} = \frac{1}{n^2(n+1)}
    \end{split}
\end{equation}

For each $m$, define $f$ on $I_{n,m}$ as follows.  From left to right, beginning at the $x$-axis, we will connect line segments to form a sequence of spikes.  We wish the first spike to point upward (above the $x$-axis) and the second downward.  To this end, consider the uniform partition of $I_{n,m} = [\alpha, \beta]$ consisting of five points: $\Part_{nm}=\{a_1, a_2, a_3, a_4, a_5\}$.  Define the following values of $f$ to mark the three anchor points (lying on the $x$-axis) and two tips of the two spikes in $[\alpha, \beta]$, respectively.
\[ 
\begin{split}
f(a_1) &= f(a_3) = f(a_5) = 0 \\
f(a_2) &= \frac{1}{\sqrt{n+1}}, ~f(a_4) = \frac{-1}{\sqrt{n+1}}.
\end{split}
\]
Connect the point $(a_i, f(a_i))$ to $(a_{i+1}, f(a_{i+1}))$ with a line segment for each $i = 1, 2, 3, 4$.  Drawing these segments for each $I_{n,m}$ completes the function's definition on $[0,1]$.  The first five generations of the graph are shown in Figure \ref{goffman_func}.

\begin{figure}[H]
    \centering
    \includegraphics[width=1\linewidth]{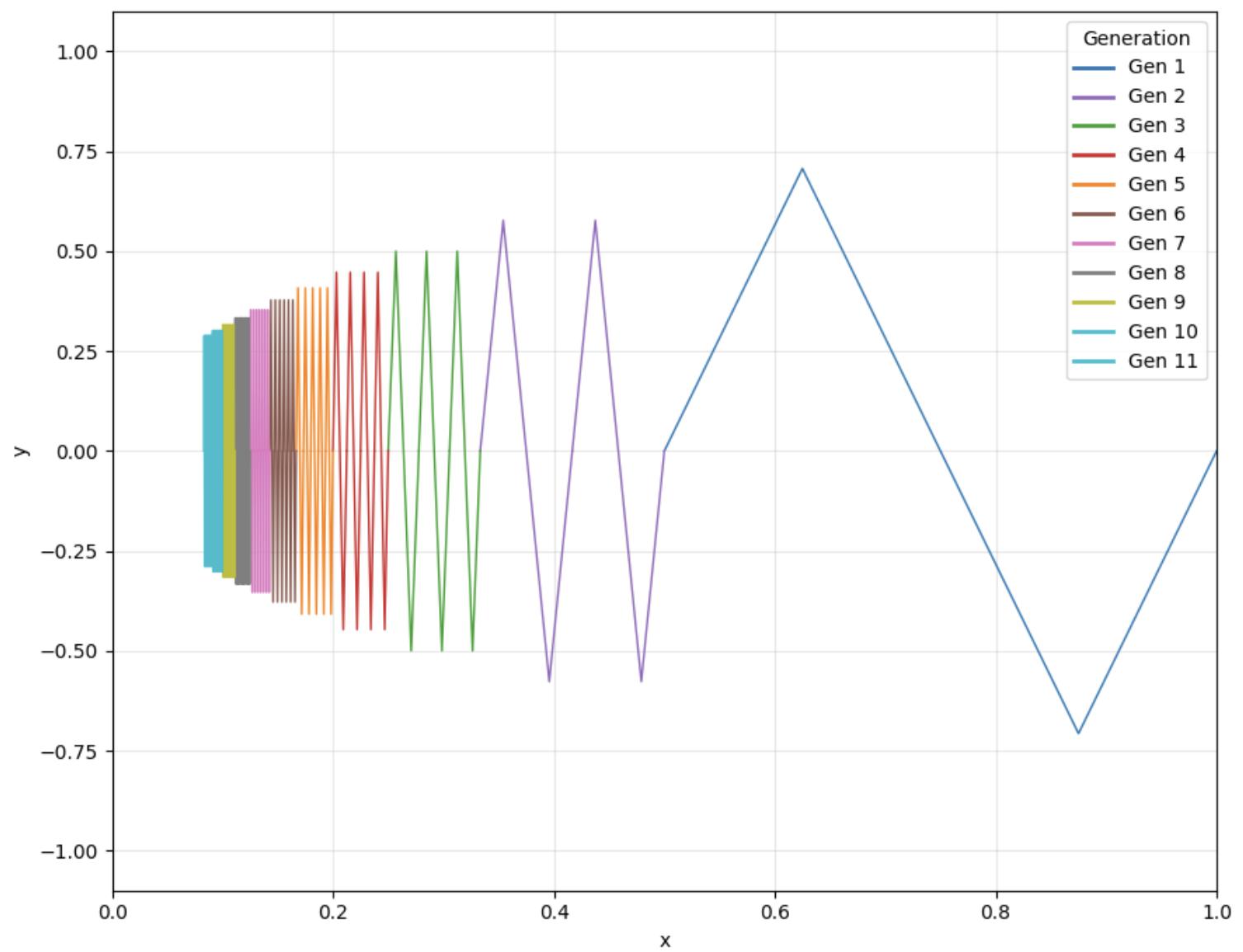}
    \caption{The first several generations of Goffman's function $f$ on $[0,1]$ with $2n$ spikes ($n$ zig-zags) defined on each $I_n$.}
    \label{goffman_func}
\end{figure}

\begin{remark}
    Goffman states that the graph of $f$ is nonporous with focus on the origin $(0,0)$, but the argument is concluded ``by studying the graph," and no graph is provided.  Also, in the definition of $I_{m,n}$, the length of $I_{m,n}$ should be 
    \[ |I_{n,m}| = 1/(n^2 (n+1))\] 
    because $I_n$ has length $1/(n(n+1))$.
\end{remark}

\section{Main result}\label{main_section}
These iteratively constructed functions are routinely viewed as power series built from continuous piecewise linear functions $g: \R \rightarrow \R$ supported on $[0,1]$, e.g. tent functions.  Such $g$ are Lipschitz and hence $\alpha$-H\"older for every $\alpha \in (0,1]$.  To build an $\alpha$-H\"older function $h:[0,1] \rightarrow \R$, for each $j \in \N$ compress the graph of $g$, translate the supported portion to the interval $(a_j, b_j)$, and rescale as desired.  The following lemma allows us to join such curves in a favorable manner.

\begin{lemma}\label{sum_spikes_lemma}
Fix $\alpha \in (0, 1]$ and suppose that $g : \mathbb{R} \to \mathbb{R}$ is an $\alpha$-H\"older function with $\text{supp } g \subset [0, 1]$. Suppose also that the intervals $[a_j, b_j]$, $j = 1, 2, \ldots$ have disjoint interiors, and the sequence $\{c_j\}$ satisfies $|c_j| \leq M(b_j - a_j)^\alpha$ for some fixed $M > 0$. Then the function
\begin{equation}\label{sum_of_spikes}
    h(x) = \sum_{j=1}^{\infty} c_j g\left(\frac{x - a_j}{b_j - a_j}\right)
\end{equation}
is $\alpha$-H\"older continuous.
\end{lemma}

\begin{proof}
By assumption there exists a constant $C$ such that
\begin{equation}\label{g_holder_prop}
    |g(x) - g(y)| \leq C|x - y|^\alpha \quad \forall x, y \in \mathbb{R}.
\end{equation}
We claim that
\begin{equation}
    |h(x) - h(y)| \leq MC|x - y|^\alpha \quad \forall x, y \in \mathbb{R}.
\end{equation}

\begin{enumerate}
    \item[\textit{Case 1:}] Both $x$ and $y$ are in the same interval $[a_j, b_j]$. Then \eqref{sum_of_spikes} and \eqref{g_holder_prop} imply
    \[
    |h(x) - h(y)| \leq |c_j| C \left(\frac{|x - y|}{b_j - a_j}\right)^\alpha \leq MC|x - y|^\alpha
    \]
    since only the $j$-th term of \eqref{sum_of_spikes} is nonzero on $[a_j, b_j]$.

    \item[\textit{Case 2:}] $x \in [a_m, b_m]$ and $y \in [a_j, b_j]$ where we may assume $b_m \leq a_j$. Notice $g(1) = g(0) = 0$, so Case 1 yields $|h(x)| = |h(x) - h(b_m)| \leq MC|x - b_m|^\alpha$. Similarly, $|h(y)| = |h(y) - h(a_j)| \leq MC|y - a_j|^\alpha$. Using the triangle inequality and the subadditivity of the function $t \mapsto t^\alpha$, we obtain
    \begin{equation*}
        \begin{split}
            |h(x) - h(y)| &\leq |h(x)| + |h(y)| \\
            &\leq MC(|x - b_m| + |y - a_j|)^\alpha \\
            &\leq MC|x - y|^\alpha,
        \end{split}
    \end{equation*}
    where the last step is based on the fact that $x \leq b_m \leq a_j \leq y$.

    \item[\textit{Case 3:}] One or both of $x, y$ is outside of $\bigcup_j [a_j, b_j]$.  If both $x$ and $y$ are outside of $\bigcup_n [a_j, b_j]$ then $h(x) = h(y) = 0$.  If $x \in [a_j, b_j]$ for some $j$ and $y$ lies outside all of these intervals, then $h(y) = 0 = h(a_j) = h(b_j)$.  There are two possibilities: $x < y$ and $x \geq y$.  If $x < y$ then $|h(x) - h(y)| = |h(x) - h(b_j)| \leq MC|x - b_j|^{\alpha} \leq MC|x-y|^{\alpha}.$  Similarly, for $x \geq y$ replacing $b_j$ with $a_j$ yields the same result. \qedhere
\end{enumerate}
\end{proof}

Adjusting Goffman's approach to accommodate any H\"older exponent $\alpha \in (0,1)$, we will employ tents with stakes at $1/n^q, ~q\in \N$, instead of zig-zags staked at $1/n$.  Proof of the main result relies on partitioning a small interval $[0,r]$ in which each subinterval must contain at least two stakes (one tent).  The following lemma demonstrates this desired behavior.
\begin{lemma}\label{two_in_one}
    For $n \in \N$ and $q > 0$ let $x_n = 1/n^q$.  For every $N \in \N$ there is $L \in \N$ such that $n \geq L$ implies 
    \[ 
    \frac{x_n}{N} > 2(x_n - x_{n+1}).
    \]
\end{lemma}
\begin{proof}
    By the mean value theorem 
    \[
    x_n - x_{n+1} \leq \frac{q}{(n+1)^{q+1}} \leq \frac{q}{n^{q+1}}.
    \]
    For $n > 2qN$ we have
    \[
    \frac{x_n - x_{n+1}}{x_n} \leq \frac{q/(n^{q+1})}{1/n^q} = \frac{q}{n} < \frac{1}{2N},
    \]
    so $L = \ceil{2qN + 1}$ is sufficiently large.
\end{proof}

To disprove porosity of $Gr(h)$ for certain continuous $h$, we draw inspiration from the definition of Assouad dimension (Definition \ref{Assouad_dim_def}), employing balls of the form $B(0, r_n/n) \subset B(0, r_n) \subset \R$.

\begin{lemma}\label{r_cubes_lemma}
    Let $h:[0,1] \rightarrow \R$ be continuous and suppose that for every $n \in \N$, there is $r_n > 0$ such that 
    \[h\left( \left[ \frac{(j-1)r_n}{n}, \frac{jr_n}{n} \right] \right) \supset [0,r_n] ~~\text{for}~~ j = 1, \dots, n.\]
    Then $Gr(h)$ is nonporous.
\end{lemma}
\begin{proof}
    Put $Q_n = [0,r_n]^2$ and note that any (and hence the largest) sub-cube $Q_n' \subset Q_n$ that does not meet $Gr(h)$ has side length $l(Q_n') \leq 2r_n/n = 2l(Q_n)/n$.  Therefore porosity is untenable in light of Lemma \ref{cubes_lemma}.
\end{proof}

\begin{theorem}
    For any $\alpha \in (0,1)$ there is an $\alpha$-H\"older function $h:\R \rightarrow \R$ such that $Gr(h)$ is nonporous.
\end{theorem}
\begin{proof}
    First we construct $h$.  Let $0 < \varepsilon < 1, ~q = 1/\varepsilon,$ and $0 < p < 1$.  Put $x_n = 1/n^q, ~n \in \N$ and $[a_n, b_n] = [x_{n+1}, x_n]$.  Consider $g(x) = 2x \chi_{[0, 1/2]} + (2 - 2x)\chi_{(1/2, 1]}$ and note $Gr(g)$ forms a tent of height $1$ on $[0,1]$.  Put $h(x) = \sum_{n = 0}^{\infty} x_n^p g_n(x)$ where 
    \[
    g_n(x) = g\left( \frac{x - a_n}{b_n - a_n}\right).
    \]
    
    To establish H\"older continuity it suffices to show $h$ satisfies Lemma \ref{sum_spikes_lemma}.  To this end, note that
    \begin{equation*}
            \left(\frac{1}{n^q}\right)^{1 + \varepsilon} \leq \frac{1}{n^{q+1}} ~~\text{is equivalent to}~~ q \geq \frac{1}{\varepsilon},
    \end{equation*}
    from which it follows that
    \begin{equation}
        \begin{split}
            x_n^{1 + \varepsilon} &\leq \frac{1}{n^{q+1}} \leq 2^{q+1}(x_n - x_{n+1}) \\
            x_n^p &\leq \left[ 2^{q+1}(x_n - x_{n+1}) \right]^{p/(1 + \varepsilon)} \\
            &= 2^{pq} (x_n - x_{n+1})^{p/(1 + \varepsilon)}.
        \end{split}
    \end{equation}
    Because $\varepsilon, p \in (0,1)$ were chosen arbitrarily, the desired H\"older exponent $\alpha = p/(1 + \varepsilon)$ of $h$ is thus obtained via Lemma $\ref{sum_spikes_lemma}$ with $c_n = x_n^p$ and $M = 2^{pq}$.

    It remains to show $Gr(h)$ is nonporous.  Lemma \ref{two_in_one} says that for any $N \in \N$ we may choose $n$ large enough that subdividing $[0, x_n]$ into $N$ subintervals of equal length puts at least two elements of $\{x_m\}_{m=n}^{\infty}$ into each subinterval.  So for $j = 1, \dots, N$ there is at least one tent supported entirely on \[ F_j = [(j-1)x_n/N, jx_n/N].\]  We will show that there is a tent of height at least $x_n$ defined within $[0, x_n/N]$, from which it follows that there are tents of height at least $x_n$ defined within every $F_j.$

    Fix $N \in \N$.  Because $p < 1$ we may then choose $K$ large enough that all of the following conditions are met:
    \begin{enumerate}[(i)]
        \item $K \geq \ceil{2qN + 1}$
        \item $N/K^q < 1/K^{qp}$
        \item \begin{equation}\label{good_n}
        \frac{N}{K^q} < \frac{1}{n^q} < \frac{1}{K^{qp}}.
    \end{equation}
    \end{enumerate}
    Note that 
    \begin{equation*}
        \left|\frac{K}{N^{1/q}} - K^p \right| \rightarrow \infty ~~\text{as} ~~ K \rightarrow \infty,
    \end{equation*} so there is indeed an integer $n$ satisfying \eqref{good_n} for sufficiently large $K$.
    In particular $1/K^q < x_n/N$, and hence the tent on $[x_{K+1}, x_K] \subset [0, x_n/N] = F_1$ has height $x_K^p > x_n$.  Clearly $h(F_j)$ also contains $[0, x_n]$ for $j>1$ because $x_n$ is decreasing, so Lemma \ref{r_cubes_lemma} shows that $Gr(h)$ is nonporous.
\end{proof}
So $\mathcal{H} = \{h_{p, \varepsilon} ~|~ \varepsilon, p \in (0,1) \}$ is a two-parameter family of H\"older continuous functions $\R \rightarrow \R$, where $h_{p, \varepsilon}$ has exponent $\alpha = p/(1+\varepsilon)$, which is actually sharp.
\begin{fact}
    $h_{p, \varepsilon}$ is not $\alpha'$-H\"older for any $\alpha' > p/(1 + \varepsilon)$.
\end{fact}
\begin{proof}
    The tent defined on $[x_{n+1}, x_n] = [1/n^{q+1}, 1/n^q]$ has height $1/n^{pq}$. The endpoints of the line segment constituting the left hand side of this tent form the ratio
    \[
    \frac{|h_{p, \varepsilon}(x_n) - h_{p, \varepsilon}(x_{n+1})|}{|x_n - x_{n+1}|^{\alpha}} = \frac{\frac{1}{n^{pq}}}{\left( \frac{1}{n^q} - \frac{1}{(n+1)^q} \right)^{\alpha} } \approx \frac{n^{\alpha(q+1)}}{n^{pq}} = n^{\alpha(q+1) - pq},
    \]
    which is bounded if and only if $\alpha < p/(1 + \varepsilon)$ because $q = 1/\varepsilon.$
\end{proof}

\section{Examples}\label{examples_section}

\subsection{A small family of Goffman-type functions}
Goffman's zig-zag construction can be adjusted to accommodate taller spikes at a higher frequency, which of course results in ``less porous" graphs and smaller (maximal) H\"older exponents.   In fact, for any prescribed $\alpha \in (0,1/4)$, the height and frequency of spikes can be chosen to obtain an $\alpha$-H\"older Goffman-type function $h: [0,1] \rightarrow \R$ with nonporous $Gr(h)$.

Generally, for any $k \in \{0\} \cup \N$, a Goffman-type function can be defined with $n^k$ zig-zags on each $I_n$ by subdividing $I_n$ into $n^k$ subintervals $I_{n,m}$, $m = 0, 1, \dots, n^k - 1$.  Notice in this case 
\[ 
|I_{n,m}| = 1/(n^{k+1}(n+1)) ~~ \text{instead of} ~~ 1/(n^2(n+1)).
\]  
Taller spikes can be achieved for each $n$ using heights $c_{n,m} = 1/n^p$ for all $m$, where $p \in (0, 1/2]$ is constant.  Nothing is gained by using $n+1$ instead of $n$ in the definition of the height coefficients $c_{n,m}$, so we use $n$.

Goffman-type functions are readily viewed as power series built from the simple function $g: \R \rightarrow \R$ whose graph contains the three line segments connecting the points $(0,0), (1/4, 1), (1/2, 0), (3/4, -1), (1, 0).$  It contains a single ``zig-zag" with two spikes on the unit interval:
\begin{equation}\label{g_def}
    g(x) = 4x\chi_{[0, 1/4)} + (-4)(x - 1/2)\chi_{[1/4, 3/4)} + 4(x - 1)\chi_{[3/4, 1]},
\end{equation}
where $\chi_A$ is the characteristic function of $A \subset \R$. Lemma \ref{sum_spikes_lemma} can be used to complete the construction, but before we begin it is worth noting the following obvious observation.
\begin{fact}
    The function $g$ is 4-Lipschitz and therefore $\alpha$-H\"older for any $\alpha \in (0,1]$.
\end{fact}

\begin{definition}\label{goffman_family}
    Fix $k \in \N$ and $p \in (0, 1/2]$.  For each $n \in \N$ and $m = 0, 1, \dots, n^k-1$, let $I_{n,m} = [a_{n,m}, b_{n,m}]$ where 
    \begin{equation}
        I_{n,m} = \left[ \frac{1}{n+1} + \frac{m}{n^{k+1} (n+1)}, \frac{1}{n+1} + \frac{m+1}{n^{k+1} (n+1)} \right], ~~ |I_{n,m}| = \frac{1}{n^{k+1}(n+1)}.
    \end{equation}
    For each $n,m$ put
    \begin{equation}
        g_{n,m}(x) = g\left( \frac{x - a_{n,m}}{b_{n,m} - a_{n,m}} \right), ~~c_{n,m} = 1/n^p.
    \end{equation}
    where $g$ is as in \eqref{g_def}.  Then we have a family of Goffman-type functions $\mathcal{H}_{sm}$ with members $h_{k,p}: \R \rightarrow \R$ defined by
    \begin{equation*}
        h_{k, p}(x) = \sum_{n=1}^{\infty} \sum_{m=0}^{n-1} c_{n,m} g_{n,m}(x).
    \end{equation*}
\end{definition}

\begin{theorem}
    Each $h_{k,p} \in \mathcal{H}_{sm}$ is $\alpha$-H\"older continuous for $\alpha \leq p/(k+2)$, and this inequality is sharp.
\end{theorem}
\begin{proof}
    Lemma \ref{sum_spikes_lemma} can be applied directly:
    \begin{equation}\label{holder_proof}
            \frac{|c_{n,m}|}{(b_{n,m} - a_{n,m})^{\alpha}}= \frac{[ n^{k+2} + n^{k+1}]^{\alpha}}{n^p}
    \end{equation}
    is bounded if and only if $\alpha \leq p/(k + 2).$  To see the inequality is sharp, consider for each $n$ a pair of endpoints $(x_n,h_{k,p}(x)),(y_n, h_{k,p}(y))$ of a line segment in the graph of $h_{k,p}$, say $x_n,y_n \in I_{n,m}$.  Then
    \begin{equation*}
    \begin{split}
        \frac{|h_{k,p}(y) - h_{k,p}(x)|}{|y-x|^{\alpha}} = \frac{(2/n^p)}{(|I_{n,m}|/4)^{\alpha}} 
        &= \frac{2^{2\alpha + 1}[n^{k+1}(n+1)]^{\alpha}}{n^p},
    \end{split}
    \end{equation*}
    which is unbounded as $n \rightarrow \infty$ when $\alpha > p/(k+2).$
\end{proof}


\subsection{Thinness}
It is known that the graph of any Zygmund function $\R^d \rightarrow \R$ is thin for doubling measures, which includes the Lipschitz class.  It is unknown whether all H\"older continuous functions have thin graphs. Thinness was proven for Zygmund graphs via porosity, but we now see that that approach generally fails for H\"older functions.

\begin{fact}
    $Gr(h_{k,p})$ is thin for all $h_{k,p} \in \mathcal{H}_{sm}$.
\end{fact}
\begin{proof} Observe that $g_{n,m}$ as defined in \eqref{g_def} is piecewise linear with four non-horizontal segments $L_1, L_2, L_3, L_4$.  Also
\[ |g_{n,m}'(x)| = \frac{2/n^p}{|I_{n,m}|/4} = \frac{8n^{k+1}(n+1)}{n^p} = q_n \] on the interior of each $\pi_1(L_i)$, so $g_{n,m}$ is $q_n$-Lipschitz.  Let $E$ and $E_{n,m}$ denote the graphs of $h_{k,p}$ and $g_{n,m}$ respectively.  Then for any doubling measure $\mu$ on $\R^2$,
\[ \mu(Gr(h_{k,p})) = \mu\left( \bigcup_n \bigcup_m Gr(g_{n,m}) \right) \leq \sum_n \sum_m \mu(Gr(g_{n,m})) = 0. \qedhere\]
\end{proof}

\subsection{Lack of porosity}
We now offer a constructive alternative to Lemma \ref{r_cubes_lemma} to disprove porosity.  To see that every member of $\mathcal{H}_{sm}$ has nonporous graph, it is unnecessary to consider more than one zig-zag on each $I_n$.  With only one zig-zag the spikes remain sufficiently crowded to preclude porosity near the origin.  To this end, consider $Gr(h_{0,1/2})$, $h_{0,1/2} \in \mathcal{H}_{sm}$, which has only two spikes (one zig-zag) on each $I_n$ of height $1/\sqrt{n}$.  We construct a sequence of cubes $Q_n$, each enclosing a portion of $Gr(h_{0,1/2})$, which demonstrates that porosity is untenable in light of Lemma \ref{cubes_lemma}.


\begin{figure}
    \centering
    \includegraphics[width=0.75\linewidth]{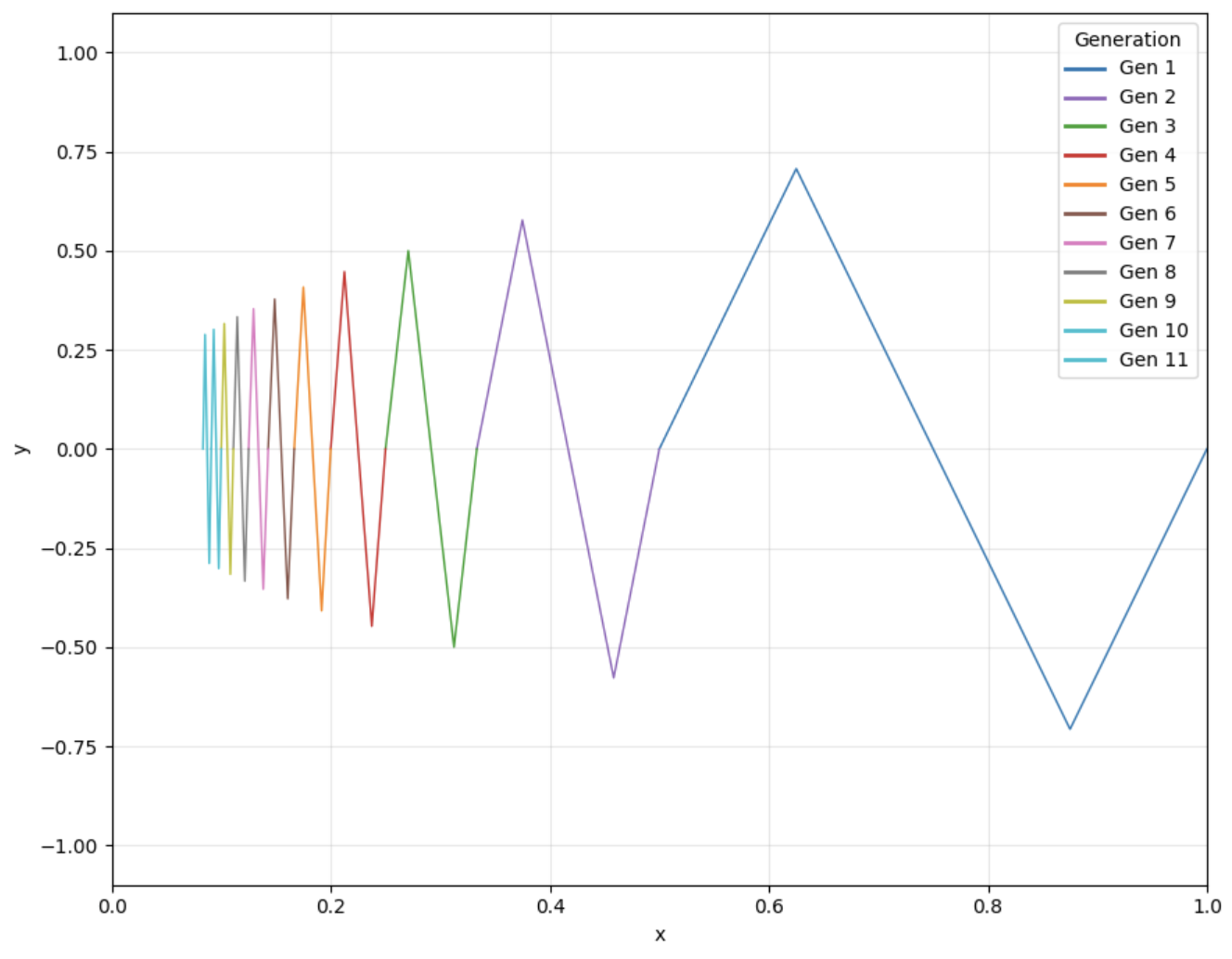}
    \caption{Simplified version $h_{0, 1/2} \in \mathcal{H}_{sm}$ of Goffman's function with only two spikes (one zig-zag) per interval $I_n$, which is nonporous.}
    \label{simplified_func}
\end{figure}

\begin{theorem}
    $Gr(h)$ is nonporous, where $h = h_{0, 1/2} \in \mathcal{H}_{sm}$.
\end{theorem}
\begin{proof}
For $n \geq 2$, consider the circles $C_n = \{(x - r_n)^2 + y^2 = r_n^2\}$ with $r_n = 1/(2n)$.  Note that every point $(x,y) \in C_n$ has $y < \sqrt{x}$ except $(0,0)$, so the height of every spike exceeds that of $C_n$.  This ensures there are no large gaps in the circle that are not due to space between spikes.  Also consider the inscribed oriented cubes $Q_n'$ of $C_n$, each with side length $l(Q_n) = 1/(\sqrt{2}n).$  Then translate $Q_n'$ to place the right hand side of the cube at $x = 1/n$.  Concretely, put 
\begin{equation}
    \begin{split}
        \alpha_n &= \frac{\diam (C_n) - l(Q_n)}{2} = \frac{1}{2}\left(\frac{1}{n} - \frac{1}{\sqrt{2}n} \right) = \frac{\sqrt{2} - 1}{2 \sqrt{2} n}, \\
        Q_n &= Q_n' + (\alpha_n, 0). \\
    \end{split}
\end{equation}

\begin{figure}
    \centering
    \begin{minipage}{0.48\textwidth}
        \centering
        \includegraphics[width=\linewidth]{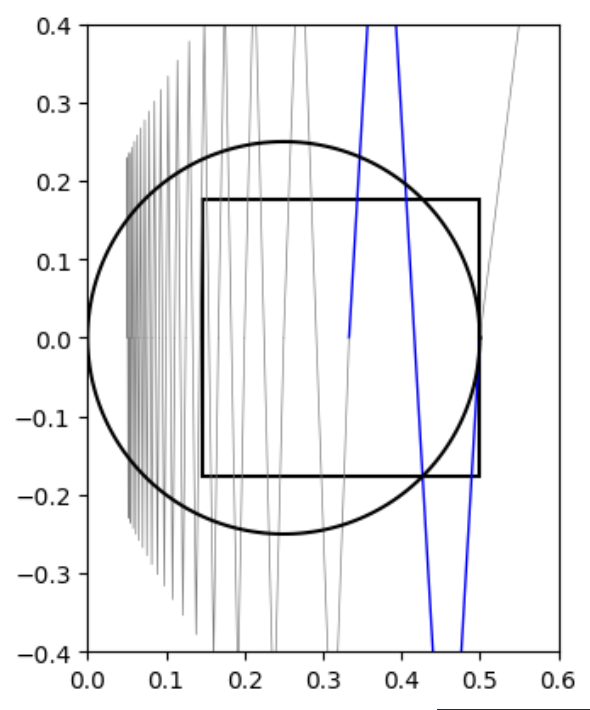}
        \caption{The second generation circle $C_2$ and the boundary of its translated inscribed cube $Q_2$.}
        \label{square1}
    \end{minipage}
    \hfill
    \begin{minipage}{0.48\textwidth}
        \centering
        \includegraphics[width=\linewidth]{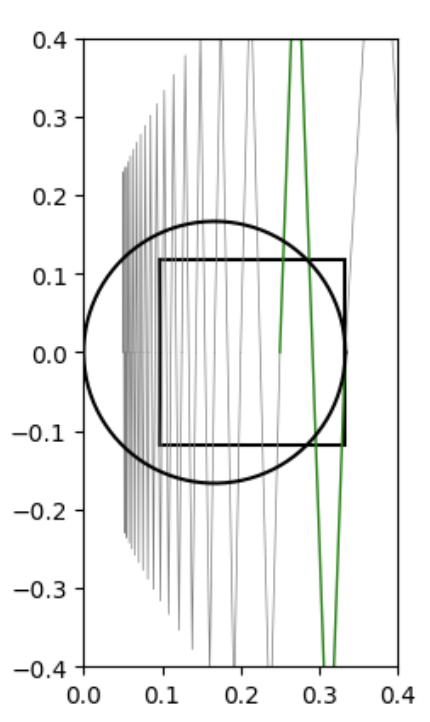}
        \caption{$C_3$ and $\partial Q_3$.}
        \label{fig:placeholder}
    \end{minipage}
\end{figure}

Let $E_n = Gr(h|_{I_n})$.  The longest horizontal line segment $S_2 \subset \overline{Q}_2$ that does not meet $E$ lies on the bottom of $\partial Q_2$ between the right-most segment of $E_3$ and the middle segment of $E_2$.  Generally, the longest such $S_n \subset \overline{Q}_n$ lies on the bottom of $\partial Q_n$ between the right-most segment of $E_{n+1}$ and the middle segment on $E_n$.

It is tedious to work with $l(S_n)$ exactly, so we note that $l(S_n) < l(S_n^{1,2})$, where the latter quantity is defined as follows.  The left-most segment of $E_n$ lies on the line $L_1$, which intersects the bottom of $\partial Q_n$, say at $v_n \in \R^2$.  Also, let us write $L_2$ for the line containing the middle segment of $E_n$.  There is a point $w_n \in L_2$ for which $\Im(w_n) = \Im(v_n)$, so we let $S_n^{1,2}$ be the line segment connecting $v_n$ to $w_n$, and hence $l(S_n^{1,2}) = |v_n - w_n|.$  Also let $L_3$ be the line containing $S_n^{1,2}.$

To see that $l(S_n) \lesssim l(Q_n)/n$, we need the equations of $L_1$ and $L_2$ which have slopes
\begin{equation}
\begin{split}\label{In_slope}
    m_1 = \frac{1 / \sqrt{n} }{ |I_n|/4 } = 4\sqrt{n}(n+1) ~~\text{ and }~~ m_2 = -m_1.
\end{split}
\end{equation}
Note that $L_1$ and $L_2$ have roots $1/(n+1)$ and $1/(n+1) + |I_n|/2$ respectively, so the three relevant linear equations are 
\begin{equation}
\begin{split}\label{consec_spike_lines}
    L_1: ~~ y &= 4\sqrt{n}(n+1) \left( x - \frac{1}{n+1}\right) \\
    L_2: ~~ y &= -4\sqrt{n}(n+1) \left( x - \frac{1}{n+1} - \frac{|I_n|}{2} \right) \\
    &= -4\sqrt{n}(n+1) \left( x - \frac{1}{n+1} - \frac{1}{2n(n+1)} \right) \\
    L_3: ~~ y &= -l(Q_n)/2 = -\frac{1}{2\sqrt{2}n}
\end{split}
\end{equation}
Recall that $\{v_n\} = L_1 \cap L_3$ and $\{w_n\} = L_2 \cap L_3$.  After tedious computations performed by Wolfram Alpha, equations \eqref{consec_spike_lines} show
\[
    \Re(v_n) = \frac{16 + \frac{\sqrt{2}}{n^{3/2}}}{16n + 16}, ~~ \Re(w_n) = \frac{16 +\frac{8}{n} - \frac{\sqrt{2}}{n^{3/2}}}{16n + 16}
\]
Therefore
\begin{equation}
    \begin{split}\label{lS}
        l(S_n^{1,2}) = \Re(w_n) - \Re(v_n) &= \frac{1}{16n + 16} \left( \frac{8}{n} - \frac{2\sqrt{2}}{n^{3/2}} \right)\\
        &\lesssim \frac{1}{n^2}.
    \end{split}
\end{equation}
It now follows from \eqref{lS} and the fact that $l(Q_n) = 1/(\sqrt{2}n)$ that 
\begin{equation*}
    \begin{split}
        l(S_n) < l(S_n^{1,2}) \lesssim \frac{1}{n^2} = \frac{\sqrt{2} l(Q_n)}{n}.
    \end{split}
\end{equation*}
Lemma \ref{cubes_lemma} now provides the desired conclusion because any cube $Q \subset (Q_n \setminus Gr(h)$ must have $l(Q) \leq l(S_n)$.
\end{proof}

\begin{corollary}
    $Gr(h_{k,p})$ is nonporous, where $h_{k,p}\in \mathcal{H}_{sm}$.
\end{corollary}
\begin{proof}
    The graph of $h_{0, 1/2}$ is sparser than that of any other $h_{k,p} \in \mathcal{H}$ because the spikes are taller and more frequent.  Because the graph of $h_{0,1/2}$ is nonporous, the graph of any $h_{k,p}$ is clearly also nonporous.
\end{proof}

The following question remains open.
    \begin{question}
	Is the graph of every H\"older function $f:[0,1] \rightarrow \R$ thin for doubling measures?
    \end{question}

\begin{acknowledgement*}
	The author thanks Leonid Kovalev for suggesting significant improvements along with Lemma \ref{sum_spikes_lemma} and references \cite{CT,Fraser,Goffman}.
\end{acknowledgement*}


\bibliographystyle{plain}
\bibliography{refs}

@article{ChenWen,
  author    = {Chen, Changhao and Wen, Shengyou},
  title     = {On thin carpets for doubling measures},
  journal   = {Proc. Amer. Math. Soc.},
  volume    = {147},
  year      = {2019},
  number    = {8},
  pages     = {3439--3449},
  issn      = {0002-9939},
  doi       = {10.1090/proc/14493},
}

@article {CT,
    AUTHOR = {Chrontsios-Garitsis, Efstathios-K. and Tyson, Jeremy T.},
     TITLE = {On the {A}ssouad spectrum of {H}\"older and {S}obolev graphs},
   JOURNAL = {Math. Proc. Cambridge Philos. Soc.},
  FJOURNAL = {Mathematical Proceedings of the Cambridge Philosophical
              Society},
    VOLUME = {180},
      YEAR = {2026},
    NUMBER = {1},
     PAGES = {105--131},
      ISSN = {0305-0041,1469-8064},
   MRCLASS = {28A78 (26A16 26A27 26A46 46E35)},
  MRNUMBER = {5004052},
       DOI = {10.1017/S0305004125101527},
       URL = {https://doi-org.ezproxy.rit.edu/10.1017/S0305004125101527},
}

@article{DiMarco,
  author  = {DiMarco, Claudio A.},
  title   = {Zygmund graphs are thin for doubling measures},
  journal = {J. Math. Anal. Appl.},
  volume  = {532},
  year    = {2024},
  number  = {1},
  pages   = {Paper No. 127954, 7},
  issn    = {0022-247X},
  doi     = {10.1016/j.jmaa.2023.127954},
}

@book {Fraser,
    AUTHOR = {Fraser, Jonathan M.},
     TITLE = {Assouad dimension and fractal geometry},
    SERIES = {Cambridge Tracts in Mathematics},
    VOLUME = {222},
 PUBLISHER = {Cambridge University Press, Cambridge},
      YEAR = {2021},
     PAGES = {xvi+269},
      ISBN = {978-1-108-47865-6},
   MRCLASS = {28-02 (28A78 28A80)},
  MRNUMBER = {4411274},
MRREVIEWER = {Tushar\ Das},
       DOI = {10.1017/9781108778459},
       URL = {https://doi-org.ezproxy.rit.edu/10.1017/9781108778459},
}

@article {FY,
    AUTHOR = {Fraser, Jonathan M. and Yu, Han},
     TITLE = {New dimension spectra: finer information on scaling and
              homogeneity},
   JOURNAL = {Adv. Math.},
  FJOURNAL = {Advances in Mathematics},
    VOLUME = {329},
      YEAR = {2018},
     PAGES = {273--328},
      ISSN = {0001-8708,1090-2082},
   MRCLASS = {28A80 (26A21 28A78 30L05)},
  MRNUMBER = {3783415},
MRREVIEWER = {Jeremy\ T.\ Tyson},
       DOI = {10.1016/j.aim.2017.12.019},
       URL = {https://doi-org.ezproxy.rit.edu/10.1016/j.aim.2017.12.019},
}

@article{Garnett,
  author  = {Garnett, John and Killip, Rowan and Schul, Raanan},
  title   = {A doubling measure on $\mathbb{R}^d$ can charge a rectifiable curve},
  journal = {Proc. Amer. Math. Soc.},
  volume  = {138},
  year    = {2010},
  number  = {5},
  pages   = {1673--1679},
  issn    = {0002-9939},
  doi     = {10.1090/S0002-9939-10-10234-2},
}

@article{Goffman,
  author    = {Goffman, Casper},
  title     = {Porous sets and convergence of {Fourier} series},
  journal   = {Proc. Amer. Math. Soc.},
  volume    = {123},
  year      = {1995},
  number    = {12},
  pages     = {3701--3703},
  issn      = {0002-9939},
  doi       = {10.2307/2161896},
}

@book{Heinonen,
  author    = {Heinonen, Juha},
  title     = {Lectures on analysis on metric spaces},
  series    = {Universitext},
  publisher = {Springer-Verlag, New York},
  year      = {2001},
  pages     = {x+140},
  isbn      = {0-387-95104-0},
  doi       = {10.1007/978-1-4613-0131-8},
}

@article {Luukkainen2,
    AUTHOR = {Luukkainen, Jouni},
     TITLE = {Assouad dimension: antifractal metrization, porous sets, and
              homogeneous measures},
   JOURNAL = {J. Korean Math. Soc.},
  FJOURNAL = {Journal of the Korean Mathematical Society},
    VOLUME = {35},
      YEAR = {1998},
    NUMBER = {1},
     PAGES = {23--76},
      ISSN = {0304-9914,2234-3008},
   MRCLASS = {54F45 (28A75)},
  MRNUMBER = {1608518},
}

@article{Luukkainen,
  author  = {Luukkainen, Jouni and Saksman, Eero},
  title   = {Every complete doubling metric space carries a doubling measure},
  journal = {Proc. Amer. Math. Soc.},
  volume  = {126},
  year    = {1998},
  number  = {2},
  pages   = {531--534},
  issn    = {0002-9939},
  doi     = {10.1090/S0002-9939-98-04201-4},
}

@article{Ojala,
  author  = {Ojala, Tuomo and Rajala, Tapio},
  title   = {A function whose graph has positive doubling measure},
  journal = {Proc. Amer. Math. Soc.},
  volume  = {144},
  year    = {2016},
  number  = {2},
  pages   = {733--738},
  issn    = {0002-9939},
  doi     = {10.1090/proc12748},
}

@article{Ojala2,
  author  = {Ojala, Tuomo and Rajala, Tapio and Suomala, Ville},
  title   = {Thin and fat sets for doubling measures in metric spaces},
  journal = {Studia Math.},
  volume  = {208},
  year    = {2012},
  number  = {3},
  pages   = {195--211},
  issn    = {0039-3223},
  doi     = {10.4064/sm208-3-1},
}

@article{Wu,
  author  = {Wu, Jang-Mei},
  title   = {Null sets for doubling and dyadic doubling measures},
  journal = {Ann. Acad. Sci. Fenn. Ser. A I Math.},
  volume  = {18},
  year    = {1993},
  number  = {1},
  pages   = {77--91},
  issn    = {0066-1953},
}

\end{document}